\documentclass[12pt]{article}
\usepackage{latexsym,amsmath,amsthm,verbatim,ifthen,amssymb,graphicx,color,mathrsfs}
\usepackage[all]{xy}
\usepackage{color}
\SelectTips{cm}{}

\newtheorem{theorem}{Theorem}[section]

\newtheorem{corollary}[theorem]{Corollary}
 
 \newtheorem{proposition}[theorem]{Proposition}
 \theoremstyle{definition}
 \newtheorem{definition}[theorem]{Definition}
 \theoremstyle{remark}
 \newtheorem{remark}[theorem]{Remark}
 \newtheorem{example}[theorem]{Example}
 \numberwithin{equation}{section}

\newcommand{\lib}{\mathbb L}
\newcommand{\bq}{\mathbb Q}
\newcommand{\br}{\mathbb R}

  \def\Aut{\operatorname{Aut}}
\def\im{\operatorname{Im}}
\begin{document}
\title
{Higher order Whitehead products and $L_\infty$ structures  on the homology of a DGL}

\author{Francisco Belch\'\i, Urtzi Buijs,  Jos\'e M. Moreno-Fern\'andez\\ and Aniceto Murillo\footnote{The authors have been supported by the MINECO grant MTM2013-41762-P and  by the Junta de
Andaluc\'\i a grant FQM-213. The second author is also supported by the MINECO grant RYC-2014-16780. The fourth author is also supported by the Vicerrectorado de Investigaci\'on of the University of M\'alaga.\vskip 1pt 2010 Mathematics subject
classification: 17B55, 18G55, 55P62.\vskip
 1pt
 Key words and phrases: Higher Whitehead product. $L_\infty$-algebra. Rational homotopy theory.}}

%\author[F. Belch\'\i]{} \address{Departamento de Algebra, Geometr\'{\i}a y Topolog\'{\i}a, Universidad de M\'alaga, Ap. 59, 29080 M\'alaga, Spain} \email{frbegu@gmail.com} \author[U. Buijs]{} \address{Departamento de Algebra, Geometr\'{\i}a y Topolog\'{\i}a, Universidad de M\'alaga, Ap. 59, 29080 M\'alaga, Spain} \email{ubuijs@uma.es} \author[J. M Moreno-Fern\'andez]{} \address{Departamento de Algebra, Geometr\'{\i}a y Topolog\'{\i}a, Universidad de M\'alaga, Ap. 59, 29080 M\'alaga, Spain} \email{morenofdezjm@gmail.com} \author[A. Murillo]{Aniceto Murillo} \address{Departamento de Algebra, Geometr\'{\i}a y Topolog\'{\i}a, Universidad de M\'alaga, Ap. 59, 29080 M\'alaga, Spain} \email{aniceto@uma.es} \thanks{} \keywords{Higher Whitehead product; $L_\infty$-algebra; rational homotopy theory} \subjclass[2010]{Primary: 55P62; Secondary: 54C35}}

\maketitle

\begin{abstract}
We detect higher order Whitehead products on the homology $H$ of a differential graded Lie algebra $L$ in terms of higher brackets in the transferred $L_\infty$ structure on $H$ via a given homotopy retraction of $L$ onto $H$.
\end{abstract}

\section{Introduction}

Topological higher order Whitehead products were introduced in \cite{Porter}: given simply connected spheres $S^{n_1},\dots,S^{n_k}$, denote by $W=S^{n_1}\vee\dots\vee S^{n_k}$ and $T=T(S^{n_1},\dots,S^{n_k})$ their wedge and fat wedge respectively. Then, there is an attaching map (in what follows we shall not distinguish a map from the homotopy class that it represents) $\omega\colon S^{N-1}\to T$ with $N=n_1+\cdots+n_k$, for which
$$
S^{n_1}\times\dots\times S^{n_k}=T\cup_\omega e^{N}.
$$
Given homotopy classes $x_j\in \pi_{n_j} (X)$, for $j=1,\dots,k$, consider the  induced map $g=(x_1,\dots,x_k)\colon W\to X$ and define
the $k$th {\em order Whitehead product set} $[x_1,\dots , x_k]_W\subset \pi_{N-1}(X)$ as the (possibly empty) set
$$\{{f\circ\omega} \ |\ f\colon T\to X\ \text{an extension of}\ g\}.$$
$$
\xymatrix{
&W\ar@{^{(}->}[d]\ar[r]^g&X\\S^{N-1}\ar[r]^\omega &T\ar@{-->}[ru]_f
}
$$
This homotopy invariant set is not only the Eckmann-Hilton dual of Massey products but
the identification of homotopy classes as higher Whitehead products has proven recently to be essential in different settings. For instance, they lie in fundamental results  used in \cite{gthe1,iki1} to explicitely exhibit the homotopy type of certain polyhedral products. Moreover, in toric topology, describing by means of higher and iterated Whitehead products  maps between polyhedral products induced by topological operations is an important question \cite{gthe2,iki2}. On the other hand, the cellular structure of well studied spaces (other than the fat wedge of spaces, of course) are described by higher Whitehead attachments. This is the case for the cellular decomposition $*\subset X^n\subset X^{2n}\subset\dots\subset X^{mn}$ in \cite{sal}   of the (ordered)  configuration space of $m$ particles on $\br^{n+1}$, $n\ge 2$.

Using the Quillen approach to rational homotopy theory \cite{Quillen} this construction has an explicit translation to the homotopy category of differential graded Lie algebras (DGL's henceforth).  Accordingly, given $L$ a DGL and classes $x_1,\dots, x_k\in H(L)$, the {\em higher order Whitehead bracket} set $[x_1, \dots , x_k]_W$\break$\subset H(L)$ is defined in a purely algebraic way \cite[\S V.2]{Tanre}, see next section for details.

On the other hand, it is well known that, given $L$ any DGL, there is a structure of minimal $L_\infty$-algebra on $H=H(L)$, unique up to $L_\infty$ isomorphism, for which $L$ and $H$ are quasi-isomorphic, as $L_\infty$-algebras. This structure is inherited from $L$ via the homotopy transfer theorem, see for instance \cite{LV}. For it, $H$ has to be presented as a linear homotopy retract of $L$ and the higher brackets $\{\ell_i\}_{ i\ge 2}$ on the $L_\infty$ structure depend, in general, on the chosen homotopy retraction.

Detecting whether a $k$th  bracket on the $L_\infty$ structure on $H$ produces a  Whitehead bracket of order $k$ becomes a good tool and  not only to treat rationally the above mentioned topological problems.  For instance, it is well known that a DGL $L$ is formal is there exists a inherited $L_\infty$ structure on $H$ as above for which $\ell_n=0$, $n\ge 3$. Moreover, a necessary condition for $L$ to be formal is that the zero class be a higher Whitehead bracket of any order.

In this paper, and given $L$ any DGL, our goal will be then to detect Whitehead brackets of  order $k$ as  $k$th brackets on the induced $L_\infty$ structure on $H$. The most general assertion in this direction that we obtain is based in \cite[Thm. 4.1]{all}: given $x\in[x_1, \dots , x_k]_W$, and up to a sign, $\ell_k(x_1,\dots, x_k)=x$ modulo brackets (of the $L_\infty$ structure) of order less than or equal to $k-1$, see Proposition \ref{elprime} for a precise statement.

 To be more accurate, and in high contrast with the Eckmann-Hilton situation concerning Massey products \cite[Theorem 3.1]{Palmieri}, extra conditions are needed. We define higher Whitehead  brackets adapted to a given homotopy retract and  prove (see Theorem \ref{main1}):
\begin{theorem} For any homotopy retract of $L$ adapted to a given $x\in[x_1,\dots,x_k]_W$, and up to a sign,
  $$
  \ell_k(x_1,\dots, x_k)=x.
  $$
  \end{theorem}
A similar assertion is obtained for any homotopy retract under the vanishing of brackets of length up to $k-2$. That is (see Theorem \ref{elsegundo}):
\begin{theorem}Let  $\ell_i=0$ for $i\le k-2$ with $k\ge 3$. Then, if $[x_1,\dots,x_k]_W$ is non empty, and also up to a sign,
$$
\ell_k(x_1,\dots, x_k)\in [x_1,\dots,x_k]_W.
$$
\end{theorem}
In particular, if $[x_1,x_2,x_3]_W\not=\emptyset$, then $\ell_3(x_1,x_2,x_3)\in [x_1,x_2,x_3]_W$. Also, $\ell_4(x_1,x_2,x_3,x_4)\in [x_1,x_2,x_3,x_4]_W$
as long as this is not the empty set and the homology of $L$ is abelian.

We finish with an example which shows that adapted retracts are needed and that the above are the best possible results in this direction even for reduced DGL's.
\section{Preliminaries}
 We  assume the reader is
familiar with  the basics of higher homotopy structures being~\cite{LV} an
excellent reference. We will also rely on some known results from rational homotopy theory
for which~\cite{FHT} is now a classic reference. With the aim of fixing notation we give some definitions
and sketch some results we will need. Throughout this paper we assume that $\mathbb{Q}$ is the base field.

A {\em graded Lie algebra} is a $\mathbb{Z}$-graded vector space $L=\oplus_{p\in \mathbb{Z}}L_p$ with a bilinear product
called the Lie bracket and denoted by $[\,\,, \,]$ verifying {\em graded antisymmetry}, $[x, y] = -(-1)^{|x||y|}
[y, x]$, and {\em graded Jacobi identity},
$$
(-1)^{|x||z|}\Bigl[x, [y, z]\Bigr]
+(-1)^{|y||x|}\Bigl[ y, [z, x]\Bigr]
+(-1)^{|z||y|}\Bigl[ z, [x, y]\Bigr]
= 0,
$$
where $|x|$ denotes the degree of $x$.

A {\em differential graded Lie algebra} (DGL henceforth) is a graded Lie algebra $L$ endowed with a linear
derivation $\partial $ of degree $-1$ such that $\partial ^2=0$. It is called free if $L$ is free as a Lie
algebra, $L = \mathbb{L}(V)$ for some graded vector space $V$.
We say that $L$ is a {\em reduced} DGL if $L_p=0$ for $p\leq 0$.

The {Quillen chain functor} associates to any differential graded Lie algebra $(L,\partial)$ the differential graded coalgebra, DGC henceforth, ${\mathcal C}(L)=(\Lambda sL,\delta)$ which is the cocommutative cofree coalgebra generated by the suspension on $L$ and whose differential is given by $\delta=\delta_1+\delta_2$,
$$
\delta_1(sx_1\wedge...\wedge sx_k)=-\sum_{i=1}^k (-1)^{n_i} sx_1\wedge...\wedge s\partial x_i\wedge...\wedge sx_k,
$$
$$
\delta_2(sx_1\wedge...\wedge sx_k)=-\sum_{i<j} (-1)^{n_{ij}+|x_i|}s[x_i,x_j]\wedge sx_1...\widehat{sx}_i...\widehat{sx}_j...\wedge sx_k .
$$
Here, $n_i=\sum_{j<i}|sx_j|$ and $n_{ij}$ is the sign given by the equality $sx_1\wedge...\wedge sx_k=(-1)^{n_{ij}}sx_i\wedge sx_j\wedge sx_1...\widehat{sx}_i...\widehat{sx}_j...\wedge sx_k $.

In \cite{Quillen}, D. Quillen constructed an equivalence
$$   \begin{array}{c}
        \text{ Simply connected}\\
        \text{ spaces}\end{array}   \xymatrix{ \ar@<0.75ex>[r]^{\lambda} &
\ar@<0.75ex>[l]^{\langle -\rangle }}\begin{array}{c}
        \text{ Reduced}\\
        \text{DGL's}\end{array}$$
between the  homotopy category of simply connected rational complexes and the homotopy category of reduced differential graded
 Lie algebras.
The reduced DGL $L$ is a   {\em model} of the simply connected complex $X$ if there is a sequence of DGL quasi-isomorphisms
\begin{equation*}
L\stackrel{\simeq }{\rightarrow}\cdots \stackrel{\simeq }{\leftarrow} \lambda X.
\end{equation*}
For any model one has $H(L)\cong\pi_*(\Omega X)\otimes\bq$. If $L=(\lib ( V),\partial)$ is free  we say that it is a {\em Quillen model} of $X$. For such a model one has $H(V,\partial_1)\cong s\widetilde H(X;\bq)$ where $\partial_1\colon V\to V$ denotes the linear part of $\partial$ and $s$ denotes the suspension operator which is defined for any graded vector space $W$ by $(sW)_p=W_{p-1}$

Next, we briefly recall from \cite[\S V]{Tanre} how to read the set of higher order Whitehead products of a simply connected complex $X$ in a given  Quillen model $L$. On the one hand, following the notation in the introduction, the map $g$ is modeled by
 $$
 \varphi\colon (\lib (u_1,\dots,u_k),0)\longrightarrow  L
$$
in which $|u_j|=n_j-1$ for each  $j=1,\dots, k$,   and the class $\overline{\varphi(u_j)}$ represents the element $x_j\in \pi_{n_j}(X)$.

On the other hand,
 arguing cellularly \cite[\S V.2]{Tanre}, the inclusion $W\hookrightarrow T$ is modeled by the DGL inclusion
$$
(\lib (u_1,\dots,u_k),0)\hookrightarrow (\lib(U),\partial)
$$
in which $|u_j|=n_j-1$, $j=1,\dots, k$,
$$
U=\langle  u_{i_1\dots  i_s}\rangle,\quad 1\leq i_1<\cdots <i_s\leq k,\quad s< k,\quad |u_{i_1\dots  i_s}|=n_{i_1}+\cdots +n_{i_s}-1,
$$
 and the differential is given by
$$
\partial u_{i_1\dots  i_s}=\sum_{p=1}^{s-1}\sum_{\sigma \in \widetilde S(p, s-p)}\varepsilon(\sigma)  \Bigr[ u_{i_{\sigma (1)}\dots i_{\sigma (p)}}, u_{i_{\sigma (p+1)}\dots i_{\sigma (s)}}\Bigr],
$$
where $ \widetilde S(p, s-p)$ denotes the set of shuffle  permutations $\sigma$ such that $\sigma(1)=1$, and $\varepsilon(\sigma)$ is given by the Koszul convention.

Moreover, a Quillen model for $S^{n_1}\times\dots\times S^{n_k}$ is obtained by attaching a single generator to $\lib(U)$ in the same way. That is:
$$
(\lib(U\oplus\langle u_{1\dots k}\rangle),\partial)
$$
with $|u_{i\dots k}|=N-1$ and
\begin{equation}\label{tan}
\partial  u_{1\dots k}=\sum_{p=1}^{k-1}\sum_{\sigma \in \widetilde S(p, k-p)}\varepsilon(\sigma)  \Bigr[ u_{i_{\sigma (1)}\dots i_{\sigma (p)}}, u_{i_{\sigma (p+1)}\dots i_{\sigma (k)}}\Bigr].
\end{equation}
We denote  $\partial u_{i_1\dots i_k}=\omega$ henceforth as it encodes the homotopy class $S^{N-1}\stackrel{\omega}{\to} T$. It follows that  there is a bijective set correspondence
of homology classes
$$[x_1, \dots , x_k]_W\cong\{\overline{\phi(\omega)} \ |\ \phi\colon (\lib(U),\partial)\to L\ \text{an extension of}\ \varphi\}.$$
\begin{equation}\label{Whiteheadlie}
\xymatrix{
(\lib (u_1,\dots,u_k),0)\ar@{^{(}->}[d]\ar[r]^(.7)\varphi&L\\(\lib(U),\partial)\ar@{-->}[ru]_\phi
}
\end{equation}
At a purely algebraic level, and for any DGL, the above subset of $H(L)$ defines the $k$th {\em order Whitehead bracket set} $\left[ x_1, \dots , x_k\right]_W$ of  given homology classes $x_1,\dots,x_k\in H(L)$, see \cite[Def. V.3(2)]{Tanre}.

\bigskip

From now on, and for simplicity in the notation, we will omit the  symbol $\otimes$ in any element of a tensor algebra.

An \emph{$L_{\infty}$-algebra} $(L,\{\ell_k\})$
is a graded vector space $L$ together with  linear maps $\ell_k\colon
L^{\otimes k}\to L$ of degree $k-2$,  for $k\ge 1$, satisfying the following two conditions:
\begin{itemize}
\item[{\rm (i)}] For any permutation $\sigma$ of $k$ elements,
$$
\ell_k(x_{\sigma(1)}\ldots x_{\sigma(k)})=\varepsilon_{\sigma}\varepsilon\ell_k(x_1\ldots x_k),
$$
where $\varepsilon_{\sigma}$ is the signature of the permutation and
$\varepsilon$ is the sign given by the Koszul convention.
\item[{\rm (ii)}] The \emph{generalized Jacobi identity} holds, that is,
$$
\sum_{i+j=n+1}\sum_{\sigma\in S(i, n-i)}\varepsilon_{\sigma}\varepsilon(-1)^{i(j-1)}
\ell_{n-i}(\ell_i(x_{\sigma(1)}\ldots x_{\sigma(i)}) x_{\sigma(i+1)}\ldots x_{\sigma(n)})=0,
$$
where $S(i,n-i)$ denotes the set of $(i, n-i)$ shuffles.
\end{itemize}
Each $L_\infty$ structure in $L$  corresponds with a  differential $\delta$ in the cofree graded cocommutative coalgebra $\Lambda^+ sL$
generated by the suspension of $L$. Indeed, every $\ell_k$ determines a degree $-1$ linear map
\begin{equation}\label{otroole}
h_k=(-1)^{\frac{k(k-1)}{2}}s\circ\ell_k\circ(s^{-1})^{\otimes k}\colon\Lambda^ksL\to sL,
\end{equation}
which extends to a coderivation
$$
\delta_k\colon \Lambda^+sL\longrightarrow\Lambda^+sL
$$
decreasing the word length by $k-1$, that is, $\delta_k(\Lambda^psL)\subset \Lambda^{p-k+1}sL$ for any $p$:
\begin{equation}\label{eles}
\delta_k(sx_1\wedge...\wedge sx_p)=\sum_{i_1<\dots<i_k}\varepsilon\, h_k(sx_{i_1}\wedge...\wedge sx_{i_k})\wedge sx_1\wedge...\widehat{ sx}_{i_1}...\widehat{sx}_{i_k}...\wedge sx_p.
\end{equation}
Every differential graded Lie algebra $(L, \partial)$ is an
$L_{\infty}$-algebra by setting $\ell_1=\partial$, $\ell_2=[\,\,,\,]$ and
$\ell_k=0$ for $k>2$. The corresponding DGC structure is precisely ${\mathcal C}(L)$.

An $L_{\infty}$-algebra $(L,\{\ell_k\})$ is called
\emph{minimal} if $\ell_1=0$.
An $L_{\infty}$-morphism between  $L$ and $L'$ is a DGC morphism
 $$
f\colon (\Lambda^+sL,\delta)\longrightarrow (\Lambda^+sL',\delta'),
$$
often denoted simply by  $f\colon L\to L'$, which is encoded by a system of skew-symmetric
linear maps $f^{(k)}\colon L^{\otimes k}\to L'$ of degree $1-k$, $k\ge 1$, satisfying an infinite sequence of equations involving the brackets
$\ell_k$ and $\ell'_k$ (see for instance~\cite{Kon03}).

 An $L_{\infty}$-morphism is a \emph{quasi-isomorphism} if $f^{(1)}\colon (L,\ell_1)\to(L',\ell_1')$ is a
quasi-isomorphism of complexes.

 Given $L$ a DGL, consider the following diagram
$$
\xymatrix{ \ar@(ul,dl)@<-5.6ex>[]_K  & (L,\partial )
\ar@<0.75ex>[r]^-q & (H,0) \ar@<0.75ex>[l]^-i }
$$
in which $H=H(L)$, $i$ is a quasi-isomorphism, $qi={\rm id}_H$ and $K$ is a chain homotopy between ${\rm id}_L$ and $iq$, i.e.,  ${\rm id}_L-iq=\partial K+K\partial $. We encode this data as $(L,i,q,K)$ and call it a {\em homotopy retract of $L$}.
In this setting,
the classical {\em Homotopy Transfer Theorem}  reads \cite{LV}:

\begin{theorem} \label{HTT}
There exists an $L_{\infty}$-algebra structure $\{\ell_k\}$ on $H$, unique up to isomorphism,
    and  $L_{\infty}$ quasi-isomorphisms
$$
\xymatrix{ (L,\partial )
\ar@<0.75ex>[r]^-Q & (H,\{\ell_k\}) \ar@<0.75ex>[l]^-I }
$$
such that $I^{(1)}=i$ and $Q^{(1)}=q$. In other words, there are DGC quasi-isomorphisms extending $i$ and $q$
$$
\xymatrix{ {\mathcal C}(L)
\ar@<0.75ex>[r]^-Q & (\Lambda sH,\delta ) \ar@<0.75ex>[l]^-I }
$$
which make $(\Lambda sH,\delta )$ a quasi-isomorphic retract of the Quillen chains on $L$.
The
transferred higher brackets are given by
\begin{equation}\label{formula}
\ell_k=\sum_{T\in \mathscr{T}_k}\frac{q\circ \ell_T}{|\Aut (T)|}.
\end{equation}
$\hfill\qed$
\end{theorem}
 We describe here every item in formula (\ref{formula}). Let $\mathscr{T}_k$ be the set of isomorphism classes of directed planar binary rooted trees  with exactly $k$ leaves. For such a tree $T$ label the leaves  by $i$, each internal edge by $K$,  and each
internal vertex by $[\,\,,\,]$. This produces a
linear map
$$
\widetilde\ell_T\colon H^{\otimes k}\longrightarrow H
$$
by moving down from the leaves to the root. For example, for $k=4$, the following tree
$$
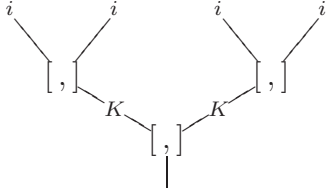
\xymatrixcolsep{1pc}
\xymatrixrowsep{1pc}
\entrymodifiers={=<1pc>} \xymatrix{
*{^i}\ar@{-}[dr] & *{} & *{^i}\ar@{-}[dl] & *{} & *{^i}\ar@{-}[dr] & & *{^i} \ar@{-}[dl]\\
*{} & {[\,,]} \ar@{-}[drr]|K & *{} & *{} & *{} & [\,,]\ar@{-}[dll]|K & *{} \\
*{} & *{} & *{} & [\,,]\ar@{-}[d] & *{} & *{} & *{} \\
*{} & *{} & *{} & *{_{\stackrel{}{}}} & *{} & *{} & *{} \\
}
$$
produces the map
$$
[\,,]\circ((K\circ [\,,]\circ(i\otimes i))\otimes (K\circ [\,,]\circ (i\otimes i))).
$$
Then,
$$
\ell_T=\widetilde\ell_{{T}}\circ \mathcal{S}_k$$
 where
 $$
 \mathcal{S}_k\colon V^{\otimes k} \to V^{\otimes k},\quad \mathcal{S}_k( v_1\ldots v_k)= \sum_{\sigma\in S_k}\varepsilon_{\sigma}\varepsilon\, v_{\sigma(1)}\ldots v_{\sigma(n)}
$$
is the symmetrization map
in which $\varepsilon_{\sigma}$
denotes the signature of the permutation and $\varepsilon$ is the sign
given by the Koszul convention.

Finally, $\Aut (T)$ stands for the automorphism group of the tree $T$.

\begin{remark}\label{remark} (i) The uniqueness property is clear. Indeed,
different homotopy retracts of $L$ produce quasi-isomorphic $L_\infty$ structures on $H$. But, since all of them are minimal, they are also isomorphic (see for instance \cite[\S4]{Kon03}).

(ii) Another invariant of transferred $L_\infty$ structures on $H$, in fact on isomorphism classes of minimal $L_\infty$ algebras is the least $k$ for which $\ell_k$ is non trivial, and the bracket $\ell_k$ itself.

\end{remark}

\section{Higher order Whitehead products and $L_\infty$ structures}

Let $(L, i, q, K)$ be a homotopy retract of a given differential graded Lie algebra $L$. The most general result relating Whitehead brackets on $H$ and brackets of the transferred $L_\infty$ structure  depends heavily on a theorem of C. Allday \cite[Thm. 4.1]{all}, see also \cite[Thm. V.7(7)]{Tanre}, and reads as follows.

\begin{proposition}\label{elprime} Let $x_1,\dots,x_k\in H$ and assume that $[x_1,\dots,x_k]_W$ is non empty. Then, for any homotopy retract of $L$ and for any $x\in [x_1,\dots,x_k]_W$,
$$
\epsilon\, \ell_k(x_1,\dots, x_k)= x+\Gamma,\quad \Gamma\in \sum_{j=1}^{k-1}\im\ell_j,
$$
where $\epsilon=(-1)^{\sum_{i=1}^{k-1}(k-i)|x_i|}$.
In particular, if $\ell_j=0$ for $j\le k-1$, then up to a sign, $\ell_k(x_1,\dots, x_k)\in  [x_1,\dots, x_k]_W$.
\end{proposition}
In the remaining of the paper $\epsilon$ will always denote the above sign.
\begin{proof} Recall that the Quillen spectral sequence of $L$ \cite{Quillen} is defined by filtering the chains ${\mathcal C}(L)$ by the kernel of the reduced diagonals, $F_p=\Lambda^{\le p}sL$. Consider the DGC quasi-isomorphisms of Theorem \ref{HTT}
$$
\xymatrix{ {\mathcal C}(L)
\ar@<0.75ex>[r]^-Q & (\Lambda sH,\delta ) \ar@<0.75ex>[l]^-I },
$$
choose the same filtration on $\Lambda sH$, and observe that at the $E^1$ level the induced morphisms of spectral sequences are both the identity on $\Lambda sH$. By comparison, all the terms in both spectral sequences are also isomorphic. Now, translating \cite[Thm. 4.1]{all}  to the spectral sequence on $\Lambda sH$ we obtain that if $[x_1,\dots,x_k]_W$ is non empty, then the element $sx_1\wedge\ldots\wedge sx_k$ survives  to  the $k-1$ page $(E^{k-1},\delta^{k-1})$. Moreover, given any $x\in [x_1,\ldots,x_k]_W $, one has
$$\delta^{k-1}\,\,\overline{sx_1\wedge\ldots\wedge sx_k}^{\,k-1}=\overline{sx}^{\,k-1}.
$$
Here $\overline{(\cdot)}^{\,k-1}$ denotes the class in $E^{k-1}$. This is to say that there exists $\Phi\in\Lambda^{\le k-1}sH$ such that
\begin{equation}\label{ole}
\delta(sx_1\wedge\ldots\wedge sx_k+\Phi)=sx.
\end{equation}
Write $\delta=\sum_{i\ge 1}\delta_i$ with each $\delta_i$ as in formula (\ref{eles}),  and decompose $\Phi=\sum_{i=2}^{k-1}\Phi_i$ with $\Phi_i\in \Lambda^isH$. By a word length argument,
$$
 \delta_k(sx_1\wedge\ldots\wedge sx_k)+\sum_{i=2}^{k-1}\delta_i(\Phi_i)=sx.
 $$
Note also that $\delta_k=h_k$ for elements of word length $k$, with $h_k$ as in (\ref{otroole}) and (\ref{eles}).  Therefore, $$h_k(sx_1\wedge\ldots\wedge sx_k)+\sum_{i=2}^{k-1}h_i(\Phi_i)=sx.$$
To finish, apply to this equation the identity (\ref{otroole}) which is equivalent to
$$
\ell_i=s^{-1}\circ h_i\circ s^{\otimes i}\quad\text{for any}\quad i\ge 1.
$$
In particular, the sign $\varepsilon$ appears when writing
$$
\ell_k(x_1,\dots, x_k)=s^{-1}\circ h_k\circ s^{\otimes k}(x_1,\dots, x_k)=\varepsilon\,s^{-1} h_k(sx_1\wedge\ldots\wedge sx_k).
$$

\end{proof}

 Next we find $k$th order Whitehead brackets that are detected precisely and only by $k$th brackets of the $L_\infty$ structure.

  Recall that, any $x\in[x_1,\dots,x_k]_W$ is produced by a DGL morphism  $\phi\colon (\lib(U),\partial)\to L$  as in diagram (\ref{Whiteheadlie}). Write,
 $$U=\langle u_1,\dots,u_k\rangle\oplus V,\quad\text{that is,}\quad
V=\langle  u_{i_1\dots  i_s}\rangle,\quad s\ge 2.
$$

 On the other hand, any homotopy retract can be obtained by decomposing  $L=A\oplus \partial A\oplus C$ with $\partial\colon A\stackrel{\cong}{\to} \partial A$ and $C\cong H$ a subspace of cycles. For it define $i\colon H\cong C\hookrightarrow L$, $q\colon L\twoheadrightarrow C\cong H$ and  $K(A)=K(C)=0$, $K\colon \partial A \stackrel{\cong}{\to} A$.

\begin{definition} With the notation above, a homotopy retract of $L$ is {\em adapted} to $x\in[x_1,\dots,x_k]_W$  if $\phi(V)\subset A$. In particular,
\begin{equation}\label{k}
K\partial\phi(u_{i_1\dots i_s})=\phi(u_{i_1\dots i_s})\quad \text{for any generator}\quad u_{i_1\dots i_s}\in V.
\end{equation}
\end{definition}

 \begin{theorem}\label{main1} Let $x\in[x_1,\dots,x_k]_W$. Then, for any homotopy retract of $L$ adapted to $x$,
  $$
  \epsilon\,\ell_k(x_1,\dots, x_k)=x.
  $$
  \end{theorem}

\begin{proof}

Let $\phi\colon (\lib(U),\partial)\to L$ with $\overline{\phi(\omega)}=x$ and consider in $H$ the $L_\infty$ structure induced by a given homotopy retract $(L,i,q,K)$ of $L$ adapted to $x$. We prove by induction on $p$, with $2\le p\le k$, that
\begin{equation}\label{ecuacion}
\phi (\partial u_{i_1\dots i_p})=\epsilon \sum_{T\in \mathscr{T}_p}\frac{1}{|\text{Aut}T|}\ell_T(x_{i_1},\dots,  x_{i_p}).
\end{equation}
The assertion is trivial for $p=2$ and assume it is satisfied for $p<k$.

Write the set  $\mathscr{T}_k$  of isomorphism classes of directed planar binary rooted trees  and exactly $k$ leaves as
$$\mathscr{T}_k=\coprod_{1\leq p\leq \lceil \frac{k}{2} \rceil } \mathscr{T}_{p, k-p},$$
 where $\mathscr{T}_{p, k-p}$ is the set of (classes of) rooted trees $T$ of the form
$$
 \xymatrixcolsep{1pc}
\xymatrixrowsep{1pc}
\entrymodifiers={=<1pc>} 
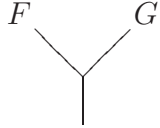
\xymatrix{
{F}\ar@{-}[rd]&&\ar@{-}[ld]{G}\\
&*{}\ar@{-}[d]&\\
&&
}$$
with $F\in \mathscr{T}_p$ and $G\in \mathscr{T}_{k-p}$. Note that, whenever $k$ is even and $p=\frac{k}{2}$ then, for any pair $F,G\in \mathscr{T}_{\frac{k}{2}}$ with $F\not= G$, the trees
$$
 \xymatrixcolsep{1pc}
\xymatrixrowsep{1pc}
\entrymodifiers={=<1pc>} 
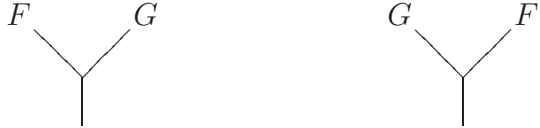
\xymatrix{
{F}\ar@{-}[rd]&&\ar@{-}[ld]{G} & && &  {G}\ar@{-}[rd]&&\ar@{-}[ld]{F}    \\
&*{}\ar@{-}[d]&     & && &  &*{}\ar@{-}[d]&\\
&& & &&&&&
}$$
are in the same class.

If $T\in \coprod_{1\leq p\le \lceil \frac{k}{2} \rceil } \mathscr{T}_{p, k-p}$, then
 ${|\text{Aut}T|}={|\text{Aut}F||\text{Aut}G|}$ except when $p= \frac{k}{2} $ and $T\in\mathscr{T}_{\frac{k}{2},\frac{k}{2}}$  is such that $F=G$. In this case, which only occurs whenever $k$ is even,  ${|\text{Aut}T|}=2|\text{Aut}F||\text{Aut}G|$.

In what follows we omit signs  to avoid excessive notation. On the one hand, splitting the summation for $p=1$, $1<p<\lceil \frac{k}{2}\rceil$ and $p=\frac{k}{2}$  (which only occurs whenever $k$ is even), we have:
\begin{align*}
&\sum_{T\in \mathscr{T}_k}\frac{1}{|\text{Aut}T|}\ell_T(x_{1},\dots,, x_{k})=\sum_{T\in \mathscr{T}_k}\frac{1}{|\text{Aut}T|}\widetilde\ell_T\circ \mathcal{S}_k (x_1,\dots, x_k)=\\
=&\sum_{T\in \mathscr{T}_k}\sum_{\sigma \in S_k}\frac{1}{|\text{Aut}T|}\widetilde\ell_T(x_{\sigma (1)},\dots, x_{\sigma (k)})\\
=&\sum_{\sigma \in S(1, k-1)}\Biggl[ix_{\sigma (1)}, K\Bigl(\sum_{T\in \mathscr{T}_{k-1}}\frac{1}{|\text{Aut}T|} \widetilde\ell_{T}(x_{\sigma (2)},\dots, x_{\sigma (k)})\Bigr)\Biggr]\\
&+\sum_{1< p< \lceil \frac{k}{2} \rceil}\sum_{\sigma \in S(p, k-p)}\Biggr[K\Bigr(\sum_{F\in \mathscr{T}_p} \sum_{\tau \in S_p}\frac{1}{|\text{Aut}F|}\widetilde\ell_F(x_{\tau \sigma (1)},\dots, x_{\tau \sigma (p)})\Bigl),
 \\
 &\hskip 3,6cm,K\Bigl(\sum_{G\in \mathscr{T}_{k-p}}\sum_{\nu \in S_{k-p}}\frac{1}{|\text{Aut}G|}\widetilde\ell_G(x_{\nu \sigma (p+1)},\dots,  x_{\nu \sigma (k)})\Bigr)\Biggr]\\
 &+\frac{1}{2}\sum_{\sigma \in S(\frac{k}{2},\frac{k}{2})}\Biggr[K\Bigr(\sum_{F\in \mathscr{T}_{\frac{k}{2}}} \sum_{\tau \in S_{\frac{k}{2}}}\frac{1}{|\text{Aut}F|}\widetilde\ell_{F}(x_{\tau \sigma (1)},\dots, x_{\tau \sigma (\frac{k}{2})})\Bigl),
 \\
 &\hskip 2,5cm,K\Bigl(\sum_{G\in \mathscr{T}_{\frac{k}{2}}}\sum_{\nu \in S_{\frac{k}{2}}}\frac{1}{|\text{Aut}G|}\widetilde\ell_{G}(x_{\nu \sigma (\frac{k}{2}+1)},\dots,  x_{\nu \sigma (k)})\Bigr)\Biggr]=(\dag)\\
\end{align*}
Note that the last summand appears only if $k$ is even. The $\frac{1}{2}$ coefficient arises from the observation above.

On the other hand, remark that the formula (\ref{tan}) can be written alternatively as
\begin{align*}
\quad\partial u_{1\dots k} =&\sum_{1\leq p< \lceil \frac{k}{2} \rceil}\sum_{\sigma \in S(p, k-p)} \Bigr[ u_{\sigma (1)\dots \sigma (p)},u_{\sigma (p+1)\dots \sigma (k)}\Bigr]\\
&+\frac{1}{2}\sum_{\sigma \in S(\frac{k}{2},\frac{k}{2})}\Bigl[  u_{\sigma (1)\dots \sigma (\frac{k}{2})}, u_{\sigma (\frac{k}{2}+1) \dots \sigma (k)} \Bigr]\\
\end{align*}
Thus, in view of equation (\ref{k}) and by induction hypothesis, we have, also modulo signs:
\begin{align*}
\,\,\phi (\partial u_{1\dots k} )=&\sum_{1\le p< \lceil \frac{k}{2} \rceil}\sum_{\sigma \in S(p, k-p)} \Bigr[ \phi (u_{\sigma (1)\dots \sigma (p)}), \phi (u_{\sigma (p+1)\dots \sigma (k)})\Bigr]\\
&+\frac{1}{2}\sum_{\sigma \in S(\frac{k}{2},\frac{k}{2})}\Bigl[ \phi( u_{\sigma (1)\dots \sigma (\frac{k}{2})}), \phi(u_{\sigma (\frac{k}{2}+1) \dots \sigma (k)} ) \Bigr]\\
=&\sum_{\sigma \in S(1, k-1)}\Bigl[\phi u_{\sigma (1)}, K\phi \partial u_{\sigma (2)\dots \sigma (k)}\Bigr]\\
&+\sum_{1< p< \lceil \frac{k}{2} \rceil}\sum_{\sigma \in S(p,k-p)}\Bigl[ K\phi \partial u_{\sigma (1)\dots \sigma (p)}, K\phi \partial u_{\sigma (p+1) \dots \sigma (k)} \Bigr]\\
&+\frac{1}{2}\sum_{\sigma \in S(\frac{k}{2},\frac{k}{2})}\Bigl[ K\phi\partial  u_{\sigma (1)\dots \sigma (\frac{k}{2})}, K\phi\partial u_{\sigma (\frac{k}{2}+1) \dots \sigma (k)}  \Bigr]\\
=&\sum_{\sigma \in S(1, k-1)}\Biggl[ix_{\sigma (1)}, K\Bigl(\sum_{T\in \mathscr{T}_{k-1}}\frac{1}{|\text{Aut}T|}\widetilde\ell_{T}\circ \mathcal{S}_{k-1} (x_{\sigma (2)},\dots, x_{\sigma (k)})\Bigr)\Biggr]\\
&+\sum_{1< p<\lceil \frac{k}{2} \rceil}\sum_{\sigma \in S(p,k-p)} \Bigg[ K\Bigl( \sum_{F \in \mathscr{T}_p}\frac{1}{|\text{Aut}F|}\widetilde\ell_F\circ \mathcal{S}_p(x_{\sigma (1)},\dots, x_{\sigma (p)})\Bigr),\\
&\hskip 2.5cm, K\Bigl( \sum_{G\in \mathscr{T}_{k-p}}\frac{1}{|\text{Aut}G|}\widetilde\ell_G\circ \mathcal{S}_{k-p}(x_{\sigma (p+1)},\dots, x_{\sigma (k)})\Bigr)\Biggr]\\
\end{align*}
\begin{align*}
&+\frac{1}{2}\sum_{\sigma \in S(\frac{k}{2},\frac{k}{2})}\Biggr[K\Bigr(\sum_{F\in \mathscr{T}_{\frac{k}{2}}} \frac{1}{|\text{Aut}F|}\widetilde\ell_{F}\circ {\mathcal S}_{\frac{k}{2}}(x_{ \sigma (1)},\dots, x_{ \sigma (\frac{k}{2})})\Bigl),
 \\
 &\hskip 2,5cm,K\Bigl(\sum_{G\in \mathscr{T}_{\frac{k}{2}}}\frac{1}{|\text{Aut}G|}\widetilde\ell_{G}\circ{\mathcal S}_{\frac{k}{2}}(x_{\sigma (\frac{k}{2}+1)},\dots,  x_{ \sigma (k)})\Bigr)\Biggr]=(\dag)\\
\end{align*}
and the assertion is proved. In particular, by the explicit formula for  $\ell_k$  in Theorem \ref{HTT},
$$
q\phi (\omega)= \epsilon \,q\sum_{T\in \mathscr{T}_k}\frac{1}{|\text{Aut}T|}\ell_T(x_1,\dots,  x_k)=\epsilon\,\ell_k(x_1,\dots, x_k).
$$
That is,
$\epsilon\,\ell_k(x_1,\dots, x_k)\in [x_1,\dots,x_k]_W$.
\end{proof}

\begin{remark} The {\em  Higher Massey products} set \cite{Massey} $\langle a_1,\dots,a_k\rangle_M\subset H^*(A)$ of order $k$ of   classes $a_1,\dots,a_k\in H^*(A)$ in the cohomology of a  given differential graded algebra  $(A,d)$ (or simply $A$) can be thought of as the ``Eckmann-Hilton'' dual of higher Whitehead brackets of order $k$. There is also an $A_\infty$ version of Theorem \ref{HTT} for a given retract of $A$,
$$
\xymatrix{ \ar@(ul,dl)@<-5.5ex>[]_K  & (A,d )
\ar@<0.75ex>[r]^-q & (H,0), \ar@<0.75ex>[l]^-i }
$$
which produces an $A_\infty$ structure $\{m_k\}$ on $H$ and $A_\infty$ quasi-isomorphisms (see for instance \cite{GS86} or \cite{HK91}):
$$
\xymatrix{  (A,d )
\ar@<0.75ex>[r]^-Q & (H,\{m_k\}). \ar@<0.75ex>[l]^-I }
$$
Recall that an $A_\infty$-algebra \cite{Stasheff} is a graded vector space $H$ endowed with  a sequence of maps $m_k\colon H^{\otimes k}\to H$ of degree $2-k$, for $k\ge 1$, satisfying a series of ``associative'' identities. Each of these maps is identified, up to suspensions and signs, with a degree $1$ map $\delta_k\colon (sH)^{\otimes k}\to sH$ which produces a differential $\delta$ on the graded colgebra  $T(sH)$. Filtering this DGC by $F_p=(sH)^{\otimes\le p}$ we obtain the {\em Eilenberg-Moore spectral sequence} from which, following the argument of Proposition \ref{elprime} now based in \cite[Thm. V.7(6)]{Tanre}, we obtain:

 If $\langle a_1,\dots,a_k\rangle_M$ is non empty, then, for any $a\in \langle a_1,\dots,a_k\rangle_M$, and any homotopy retract of $A$,
$$
\epsilon\,m_k(a_1\dots a_k)=a+\Gamma,\quad \Gamma\in \sum_{j=1}^{k-1}\im m_j.
$$
In this setting one can go a step further, see \cite[Theorem 3.1]{Palmieri}, and prove that for any homotopy retract of $A$, if $\langle a_1,\dots,a_k\rangle_M$ is non empty, then $\epsilon\, m_k(a_1,\dots,a_k)\in \langle a_1,\dots,a_k\rangle_M$.

Except for the case $k=3$ in Corollary \ref{caso3} below, this particular behavior cannot be attained in general in the $L_\infty$ setting  and additional conditions are required as the following result shows.
\end{remark}
\begin{theorem}\label{elsegundo} Let $L$ be a DGL such that, on $H$, $\ell_i=0$ for $i\le k-2$ with $k\ge 3$. If $[x_1,\dots,x_k]_W\not= \emptyset$, then
$$
\epsilon\,\ell_k(x_1,\dots, x_k)\in [x_1,\dots,x_k]_W.
$$
\end{theorem}

Observe that, in view of (ii) of Remark \ref{remark} the assumption on the vanishing of $\ell_i$ for $i\le k-2$ is independent of the chosen retract of $L$ and hence, the result remains valid for any of them.

\begin{proof}  We first observe the following: consider $(\Lambda sH,\delta)$ the DGC equivalent to the $L_\infty$ structure on $H$ given by any homotopy retract of $L$. The condition $\ell_i=0$ for $i\le k-2$ is clearly  equivalent via equations (\ref{otroole}) and (\ref{eles}) to the vanishing of the coderivations $\delta_i$ of $\Lambda^+sH$ also for $i\le k-2$. On the other hand,  as in equation (\ref{ole}) in the proof of Proposition \ref{elprime}, for any $x\in[x_1,\dots,x_k]_W$, we have.
 $$
 \delta(sx_1\wedge\ldots\wedge sx_k+\Phi)=sx\quad\text{with $\Phi\in\Lambda^{\le k-1}sH$}.
 $$
 Write again $\delta=\sum_{i\ge 1}\delta_i$ with each $\delta_i$ as in formula (\ref{eles}). Thus, a word length argument
 together with $\delta_i=0$, for $i\le k-2$, readily implies in particular that
 $$
 \delta_{k-1}(sx_1\wedge\dots\wedge sx_k)=0.
 $$
 But
 $$
 \delta_{k-1}(sx_1\wedge\ldots\wedge sx_k)=\sum_{i=1}^{k}\varepsilon\, h_{k-1}(sx_{1}\wedge...\widehat{sx}_i... \wedge sx_k)\wedge sx_i.
 $$
 Hence, via identity (\ref{otroole}),
\begin{equation}\label{masole}
 \ell_{k-1}(x_{i_1},\dots, x_{i_{k-1}})=0\quad\text{for any $\,\,1\le i_1<\dots<{i_{k-1}}\le k$}.
 \end{equation}

  Next, for each $p\le k$, let $U_p\subset U$ be the subspace generated by
$$
U_s=\langle  u_{i_1\dots  i_s}\rangle,\quad 1\leq i_1<\cdots <i_s\leq k,\quad s< p.
$$
Clearly, $(\lib (U_p),\partial)$ is a sub DGL of $(\lib(U),\partial)$ and $(\lib (U_k),\partial)=(\lib(U),\partial)$. We also denote,
$$
V_p=\langle  u_{i_1\dots  i_s}\in U_p,\,\, s\ge 2\rangle.
$$
Again $U_p=V_p\oplus \langle u_1,\dots,u_k\rangle$ and $V_k=V$.
Let $L=A\oplus \partial A\oplus C$ the decomposition giving rise to the chosen arbitrary homotopy retract.
By induction on $p$, with $3\le p\le k$, we will construct a DGL morphism $\phi\colon \lib (U_p)\to L$ for which  $\phi(V_p)\subset A$.

For $p =3$, as $[x_1,\dots,x_k]_W$ is non empty, let
 $\psi \colon (\mathbb{L}(U),\partial )\to (L, \partial )$ as in (\ref{Whiteheadlie}). We define $\phi\colon \lib(U_3)\to L$ by
 $$
 \phi(u_i)=\psi(u_i),\quad i=1,2,3.\quad\phi(u_{i_1i_2})=K\partial\psi(u_{i_1i_2}),\quad 1\le i_1<i_2\le k,
 $$
Obviously $\phi(V_3)\subset A$ and using the trivial identity for any homotopy retract
$
\partial K \partial=\partial
$ we also see that $\phi$  commutes with the differential:
$$
\partial\phi(u_{i_1i_2})=\partial K\partial \psi(u_{i_1i_2})=\partial \psi(u_{i_1i_2})=\psi\partial(u_{i_1i_2})=\phi\partial(u_{i_1i_2}).
$$
Assume the assertion true for $k-1$. That is, there exists a DGL morphism
$$
\phi\colon \lib (U_{k-1})\longrightarrow L$$ for which $\phi(V_{k-1})\subset A$. In particular, we have,
$$
K\partial\phi(u_{i_1\dots i_s})=\phi(u_{i_1\dots i_s})\quad \text{for any generator}\quad u_{i_1\dots i_s}\in V_{k-1},
$$
which is equation (\ref{k}) for $\phi$. Then, the same argument as in the proof of Theorem \ref{main1} proves the analogous of equation (\ref{ecuacion}). In particular,
$$
\phi(\partial u_{i_1\dots  i_{k-1}})= \epsilon\,\sum_{T\in \mathscr{T}_{k-1}}\frac{1}{|\text{Aut}T|}\ell_T(x_{i_1},\dots,  x_{i_{k-1}}),
$$
and therefore,
$$
q\phi(\partial u_{i_1\dots  i_{k-1}})=\epsilon\,q \sum_{T\in \mathscr{T}_{k-1}}\frac{1}{|\text{Aut}T|}\ell_T(x_{i_1},\dots,  x_{i_{k-1}})=\epsilon\,\ell_{k-1}(x_{i_1},\dots,  x_{i_{k-1}}),
$$
which is zero by the observation (\ref{masole}) above. Hence,
$$
\phi(\partial u_{i_1\dots  i_{k-1}})=\partial\Psi_{i_1\dots  i_{k-1}}
$$ and we define
$$
\phi(u_{i_1\dots  i_{k-1}})=K\partial\Psi_{i_1\dots  i_{k-1}}.
$$
Obviously $\phi(V_k)=\phi(V)\subset A$ and using again the identity  $\partial K\partial=\partial$ we see that $\phi$  commutes with differentials. Therefore $\overline{\phi(\omega)}$ is an element in $[x_1,\dots,x_k]_W$ for which we can apply  Theorem \ref{main1} and the proof is finished.

\end{proof}

\begin{corollary}\label{caso3} Let $x_1,x_2,x_3\in H$ such that  $[x_1,x_2,x_3]_W\not= \emptyset$. Then, for any homotopy retract,
$$
\epsilon\,\ell_3(x_1, x_2, x_3)\in [x_1,x_2,x_3]_W.
$$
\hfill$\square$
\end{corollary}

\begin{corollary}\label{caso4} Let $L$ be a DGL such that $H$ is  abelian. If $[x_1,x_2,x_3, x_4]_W\not= \emptyset$, then, for any homotopy retract,
$$
\epsilon\,\ell_4(x_1, x_2, x_3,  x_4)\in [x_1,x_2,x_3, x_4]_W.
$$
\hfill$\square$
\end{corollary}

We finish with an example which shows that Theorem \ref{elsegundo} is the most general version of its Eckmann-Hilton dual, even for reduced DGL's or equivalently, for simply connected rational complexes.

\begin{example}\label{contraejemplo}
%Denote by $P=S^3\times S^3\times S^3\times S^3$. If we consider the spheres as CW-complexes with one $0$-cell and one $3$-cell, then the product has a %CW-decomposition induced and the Fat wedge $T=T(S^3,S^3, S^3, S^3)$ is the
Consider the following DGL,
$$(L,\partial)=(\mathbb{L}(v_1,v_2,v_3,v_4, v_{12}, v_{13}, v_{14}, v_{23}, v_{24}, v_{34}, z, w_{123}, w_{124},v_{134}, v_{234}), \partial ),$$

where $|v_i|=2$, $1\leq i\leq 4$, $|z|=5$ and the differential is given by:
\begin{align*}
\partial (v_i)&=0,\quad 1\leq i\leq 4;\\
\partial (v_{ij})&=[v_i, v_j],\quad 1\leq i<j \leq 4;\\
\partial (z)&=0;\\
\partial(v_{ijk})&=[v_i, v_{jk}]-[v_{ij}, v_k]-[v_j, v_{ik}];\\
\partial(w_{ijk})&=[v_i, v_{jk}]-[v_{ij}+z, v_k]-[v_j, v_{ik}].\\
\end{align*}
\noindent The realization of this DGL is (of the rational homotopy type of) the complex $X$ obtained by removing two  $9$-cells from the space $T(S^3,S^3, S^3, S^3)\vee S^6$ and attach them again in a twisted way using the sphere $S^6$.

We claim that $[\overline{v}_1,\overline{v}_2,\overline{v}_3,\overline{v}_4]_W$ is non empty  and that, for any homotopy retract of $L$,
$$\ell_4(\overline{v}_1,\overline{v}_2,\overline{v}_3,\overline{v}_4)\notin [\overline{v}_1,\overline{v}_2,\overline{v}_3,\overline{v}_4]_W.$$

\medskip

For it,  define a DGL morphism  $\phi $ which solves the extension problem
$$
\xymatrix{
(\mathbb{L}(u_1,u_2,u_3, u_4),0)\ar@{^{(}->}[d]\ar[r]^(.73)\varphi&L\\
(\lib(U),\partial)\ar@{-->}[ru]_(.54)\phi
}
$$
as follows:
$$\phi(u_1)=v_1;\ \phi(u_2)=v_2;\  \phi(u_3)=v_3;\ \phi(u_4)=v_4;$$
$$\phi (u_{12})=v_{12}+z;$$
$$\phi (u_{13})=v_{13};\  \phi (u_{14})=v_{14};\
\phi (u_{23})=v_{23};\
\phi (u_{24})=v_{24};\
\phi (u_{34})=v_{34}$$
$$\phi (u_{123})=w_{123};\ \phi (u_{124})=w_{124};$$
$$\phi (u_{134})=v_{134};\ \phi (u_{234})=v_{234}.$$
Then, $[\overline{v}_1,\overline{v}_2,\overline{v}_3,\overline{v}_4]_W\not=\emptyset$. More precisely, the morphism $\phi $ defines the non zero $4$th order Whitehead bracket
$\overline{\phi(\omega)}=\overline{\Phi}$ where
\begin{align*}
\Phi&=[w_{123}, v_4]-[w_{124}, v_3]+[v_{12}, v_{34}]+[z, v_{34}]\\
&+[v_{14}, v_{23}]+[v_1, v_{234}]-[v_{13}, v_{24}]+[v_{134}, v_2].
\end{align*}
Note that $H$ is not an abelian Lie algebra and that, for any decomposition $A\oplus \partial A\oplus C$ giving rise to any chosen homotopy retract, the element $\phi (u_{12})=v_{12}+z\notin A$ as $z$ represents a non zero class. Hence, theorems \ref{main1} and \ref{elsegundo} do not apply.

In fact, it is straightforward to check that $\overline\Phi$ generates $H_{10}(L)$ and it is the only element in $[\overline{v}_1,\overline{v}_2,\overline{v}_3,\overline{v}_4]_W$. Moreover, for any homotopy retract, $\ell_4(\overline{v}_1,\overline{v}_2,\overline{v}_3,\overline{v}_4)\not= \overline\Phi $.

To help the reader with the computations, we make explicit a particular decomposition of $L$ as $A\oplus \partial A\oplus C$ with $\partial\colon A\stackrel{\cong}{\to} \partial A$ and $C\cong H$ up to degree $10$. Here, the first column denotes degree and the twisted arrow $\Rsh$ indicates the action of $\partial$ in the corresponding set.

{\footnotesize $$\begin{array}{llll}
&\hskip .8cm A&\hskip2cm \partial A&\hskip .8cm C\\
2&&&v_1,v_2,v_3,v_4\\
3&&&\\
4&&[v_1,v_2],[v_1,v_3],[v_1,v_4],&\\
&\hskip 1cm \Rsh &[v_2,v_3],[v_2,v_4], [v_3,v_4]&\\
5&v_{12},v_{13},v_{14},&&z\\
&v_{23},v_{24},v_{34}&&\\
6&&\Bigl[ v_1,[v_1,v_2]\Bigr],\Bigl[ v_1,[v_1,v_3]\Bigr],\Bigl[ v_1,[v_1,v_4]\Bigr],&\\
&&\Bigl[ v_2,[v_2,v_1]\Bigr],\Bigl[ v_2,[v_2,v_3]\Bigr],\Bigl[ v_2,[v_2,v_4]\Bigr],&\\
&&\Bigl[ v_3,[v_3,v_1]\Bigr],\Bigl[ v_3,[v_3,v_2]\Bigr],\Bigl[ v_3,[v_3,v_4]\Bigr],&\\
&&\Bigl[ v_4,[v_4,v_1]\Bigr],\Bigl[ v_4,[v_4,v_2]\Bigr],\Bigl[ v_4,[v_4,v_3]\Bigr],&\\
&&\Bigl[ v_1,[v_2,v_3]\Bigr],\Bigl[ v_2,[v_1,v_3]\Bigr],&\\
&&\Bigl[ v_1,[v_2,v_4]\Bigr],\Bigl[ v_2,[v_1,v_4]\Bigr],&\\
&&\Bigl[ v_1,[v_3,v_4]\Bigr],\Bigl[ v_3,[v_1,v_4]\Bigr],&\\
&\hfill \Rsh &\Bigl[ v_2,[v_3,v_4]\Bigr],\Bigl[ v_3,[v_2,v_4]\Bigr]&\\
7&[v_1,v_{12}],[v_1,v_{13}], [v_1,v_{14}],&&[z, v_1],[z, v_2],\\
&[v_2,v_{12}],[v_2,v_{23}], [v_2,v_{24}],&&[z, v_3],[z, v_4]\\
&[v_3,v_{13}],[v_3,v_{23}], [v_3,v_{34}],&&\\
&[v_4,v_{14}],[v_4,v_{24}], [v_4,v_{34}],&&\\
&[v_1, v_{23}], [v_2, v_{13}]&-[v_{12}, v_3]+[v_1, v_{23}]-[v_2, v_{13}]-[z,v_3]&\\
&[v_1, v_{24}], [v_2, v_{14}]&-[v_{12}, v_4]+[v_1, v_{24}]-[v_2, v_{14}]-[z,v_4]&\\
&[v_1, v_{34}], [v_3, v_{14}]&-[v_{13}, v_4]+[v_1, v_{34}]-[v_3, v_{14}]&\\
&[v_2, v_{34}], [v_3, v_{24}]&-[v_{23}, v_4]+[v_2, v_{34}]-[v_3, v_{24}]&\\
&\hskip 2.5cm \Rsh &&\\
8&w_{123}, w_{124}, v_{134}, v_{234}&\Bigl[ v_1,\Bigl[v_1,[v_1,v_2]\Bigr]\Bigr],\Bigl[ v_1,\Bigl[v_1,[v_1,v_3]\Bigr]\Bigr],\dots & \\
9& \dots & \dots &\dots\\
10& \dots & \dots &\Phi
\end{array}$$}

\end{example}

\begin{remark}  Let $X$ be a $1$-connected CW-complex of finite type and let $L$ be a reduced, finite type DGL model of $X$. If $(\Lambda sH,\delta)$ is the DGC  equivalent to a transferred $L_\infty$ structure on $H$, then, its dual $(\Lambda sH,\delta)^\sharp$ is isomorphic \cite[\S23]{FHT}  to $(\Lambda (sH)^\sharp,d)$ which is the Sullivan minimal model of $X$. In this case, Proposition \ref{elprime} and theorems \ref{main1} and \ref{elsegundo} are in some sense the reciprocal of Theorem 5.4 in \cite{anar}.

\end{remark}

\bigskip

\noindent\sc{Francisco Belch{\'\i}\\
\noindent\sc{Department of Mathematics\\
University of Southampton,  Salisbury Rd, Southampton SO17 1BJ, UK.}\\

\noindent\tt{frbegu@gmail.com}
\\

\noindent\sc{Urtzi Buijs, Jos\'e M. Moreno-Fern\'andez and Aniceto Murillo\\
\noindent\sc{Departamento de \'Algebra, Geometr{\'\i}a y Topolog{\'\i}a,\\
Universidad de M\'alaga, Ap. 59, 29080 M\'alaga, Spain.}\\

\noindent\tt{ubuijs@uma.es}
\\
\noindent\tt{morenofdezjm@gmail.com}
\\
\noindent\tt{aniceto@uma.es}

  %\author[F. Belch\'\i]{} \address{Departamento de Algebra, Geometr\'{\i}a y Topolog\'{\i}a, Universidad de M\'alaga, Ap. 59, 29080 M\'alaga, Spain} \email{frbegu@gmail.com} \author[U. Buijs]{} \address{Departamento de Algebra, Geometr\'{\i}a y Topolog\'{\i}a, Universidad de M\'alaga, Ap. 59, 29080 M\'alaga, Spain} \email{} \author[J. M Moreno-Fern\'andez]{} \address{Departamento de Algebra, Geometr\'{\i}a y Topolog\'{\i}a, Universidad de M\'alaga, Ap. 59, 29080 M\'alaga, Spain} \email{morenofdezjm@gmail.com} \author[A. Murillo]{Aniceto Murillo} \address{Departamento de Algebra, Geometr\'{\i}a y Topolog\'{\i}a, Universidad de M\'alaga, Ap. 59, 29080 M\'alaga, Spain} \email{aniceto@uma.es} \thanks{} \keywords{Higher Whitehead product; $L_\infty$-algebra; rational homotopy theory} \subjclass[2010]{Primary: 55P62; Secondary: 54C35}}
\end{document}